\documentclass[11pt,letterpaper,reqno]{amsart}

\usepackage{amsmath,amssymb,amsthm,amsfonts,mathtools}
\usepackage{mathrsfs}
\usepackage{xcolor}
\usepackage[T1]{fontenc}
\usepackage{microtype}
\usepackage{aliascnt}

\usepackage[colorlinks=true,
  linkcolor=blue,
  citecolor=blue,
  urlcolor=blue]{hyperref}
\usepackage[nameinlink,capitalize,noabbrev]{cleveref}

\allowdisplaybreaks
\numberwithin{equation}{section}

\newtheorem{theorem}{Theorem}[section]

\newaliascnt{problem}{theorem}
\newtheorem{problem}[problem]{Problem}
\aliascntresetthe{problem}

\newaliascnt{proposition}{theorem}
\newtheorem{proposition}[proposition]{Proposition}
\aliascntresetthe{proposition}

\newaliascnt{lemma}{theorem}
\newtheorem{lemma}[lemma]{Lemma}
\aliascntresetthe{lemma}

\newaliascnt{corollary}{theorem}
\newtheorem{corollary}[corollary]{Corollary}
\aliascntresetthe{corollary}

\newaliascnt{claim}{theorem}

\aliascntresetthe{claim}

\newaliascnt{question}{theorem}

\aliascntresetthe{question}

\newaliascnt{conjecture}{theorem}

\aliascntresetthe{conjecture}

\theoremstyle{definition}

\newaliascnt{definition}{theorem}

\aliascntresetthe{definition}

\newaliascnt{example}{theorem}

\aliascntresetthe{example}

\theoremstyle{remark}

\newaliascnt{remark}{theorem}
\newtheorem{remark}[remark]{Remark}
\aliascntresetthe{remark}

\theoremstyle{plain}

\crefname{theorem}{Theorem}{Theorems}
\Crefname{theorem}{Theorem}{Theorems}
\crefname{problem}{Problem}{Problems}
\Crefname{problem}{Problem}{Problems}
\crefname{proposition}{Proposition}{Propositions}
\Crefname{proposition}{Proposition}{Propositions}
\crefname{lemma}{Lemma}{Lemmas}
\Crefname{lemma}{Lemma}{Lemmas}
\crefname{corollary}{Corollary}{Corollaries}
\Crefname{corollary}{Corollary}{Corollaries}
\crefname{claim}{Claim}{Claims}
\Crefname{claim}{Claim}{Claims}
\crefname{question}{Question}{Questions}
\Crefname{question}{Question}{Questions}
\crefname{conjecture}{Conjecture}{Conjectures}
\Crefname{conjecture}{Conjecture}{Conjectures}
\crefname{definition}{Definition}{Definitions}
\Crefname{definition}{Definition}{Definitions}
\crefname{example}{Example}{Examples}
\Crefname{example}{Example}{Examples}
\crefname{remark}{Remark}{Remarks}
\Crefname{remark}{Remark}{Remarks}

\DeclareMathOperator{\rank}{rank}

\DeclareMathOperator{\sgn}{sgn}
\newcommand{\mr}{\operatorname{mr}}
\newcommand{\mrp}{\operatorname{mr}_{+}}
\newcommand{\fod}{\operatorname{fod}}
\newcommand{\KG}{\operatorname{KG}}
\newcommand{\F}{\mathbb F}
\newcommand{\R}{\mathbb R}
\newcommand{\cS}{\mathcal S}
\newcommand{\cSp}{\mathcal S_{+}}
\newcommand{\eps}{\varepsilon}

\begin{document}

\title{Minimum-rank parameters of complements of threshold Kneser graphs}

\author[T.~Hu]{Tao Hu}
\address{School of Mathematics and Statistics, Xi'an Jiaotong University, Xi'an 710049, P. R. China}
\email{hu\_tao@stu.xjtu.edu.cn}

\author[Q.~Tang]{Quanyu Tang}
\address{School of Mathematical Sciences, University of Science and Technology of China, Hefei 230026, P. R. China}
\email{tangquanyu827@gmail.com}



\begin{abstract}
Let $J_{\ge s}(n,k)$ be the graph whose vertices are the $k$-subsets of
$[n]$, with two distinct vertices adjacent whenever their intersection has
size at least $s$. Equivalently, $J_{\ge s}(n,k)$ is the complement of a
threshold Kneser graph. We determine both the symmetric minimum rank over an
arbitrary infinite field and the real positive semidefinite minimum rank of
this family. Specifically, for $k\ge2$, $1\le s\le k-1$, and
$n\ge2k-s$, we prove
$$
\mr^{\mathbb F}\left(J_{\ge s}(n,k)\right)
=
\binom{n-2(k-s)}{s}
$$
for every infinite field $\mathbb F$, and
\[
\mrp^{\mathbb R}\left(J_{\ge s}(n,k)\right)
=
\binom{n-2(k-s)}{s}.
\]
The lower bound follows from a diagonal submatrix indexed by two carefully
chosen families of $k$-subsets. For the upper bound, we construct a symmetric
matrix using an exterior power of a bilinear form, a Lagrange interpolation
identity, and a generic nonvanishing argument. Over $\mathbb R$, an interlacing
choice of parameters makes the bilinear form positive definite and yields a positive semidefinite matrix attaining the required upper bound. As consequences, we answer a
question from an American Institute of Mathematics workshop, determine the
real faithful orthogonality dimension of all graphs $J_{\ge s}(n,k)$ in the
stated range, and recover the known minimum-rank formula for Johnson graphs.
\end{abstract}

\maketitle

\section{Introduction}\label{sec:intro}

Minimum-rank problems ask for the least possible rank of a matrix whose off-diagonal zero--nonzero pattern is prescribed by a graph.  Let $G$ be a
finite simple graph with vertex set $V(G)$, and let $\F$ be a field.  We write
$\cS^\F(G)$ for the set of symmetric matrices
$A=(a_{uv})_{u,v\in V(G)}$ over $\F$ such that, for distinct vertices $u,v$,
\[
        a_{uv}\ne0
        \quad\Longleftrightarrow\quad
        \{u,v\}\in E(G).
\]
The diagonal entries are unrestricted.  The \emph{minimum rank} of $G$ over $\F$ is
\[
        \mr^\F(G):=\min\{\rank A:A\in\cS^\F(G)\}.
\]
Over $\R$, we also consider the \emph{positive semidefinite minimum rank}
\[
        \mrp^\R(G):=\min\{\rank A:A\in\cSp^\R(G)\},
        \qquad
        \cSp^\R(G):=\{A\in\cS^\R(G):A\succeq0\},
\]
where $A\succeq0$ means that $x^\top A x\ge0$ for every
$x\in\R^{V(G)}$. Thus $\mr^\R(G)\le \mrp^\R(G)$.  The first parameter is the standard
symmetric minimum rank of \(G\), and the second is its positive semidefinite
minimum rank; see, for example,
\cite{FallatHogben2007,Hogben2010}. Our
convention differs from Haemers' min-rank convention, in which matrices need
not be symmetric, diagonal entries must be nonzero, and entries corresponding
to edges may be zero \cite{Haemers1981}.

For positive integers $n$ and $k$, write $[n]=\{1,2,\ldots,n\}$ and $\binom{[n]}{k}=\{S\subseteq[n]: |S|=k\}$. For $1\le s\le k-1$, define $J_{\ge s}(n,k)$ to be the graph with vertex set $\binom{[n]}{k}$ in which two distinct vertices $S,T$ are adjacent if and only if $|S\cap T|\ge s$. Following Golovnev and Haviv \cite[Definition~3.1]{GolovnevHaviv2022}, let
$K^{<}(n,k,s)$ denote the threshold Kneser graph on
$\binom{[n]}{k}$ in which two distinct vertices are adjacent if and only if
their intersection has size less than $s$. Thus
\[
        J_{\ge s}(n,k)=\overline{K^{<}(n,k,s)}.
\]

More generally, for \(I\subseteq\{0,1,\ldots,k-1\}\), let \(J_I(n,k)\)
denote the generalized Johnson graph on \(\binom{[n]}{k}\) in which distinct
vertices \(S,T\) are adjacent if and only if \(|S\cap T|\in I\).  Thus
\[
        J_{\ge s}(n,k)=J_{\{s,s+1,\ldots,k-1\}}(n,k).
\]
At the two endpoints,
\[
        J_{\ge1}(n,k)=\overline{\KG(n,k)}
        \qquad\text{and}\qquad
        J_{\ge k-1}(n,k)=J(n,k),
\]
where \(\KG(n,k)\) is the Kneser graph and \(J(n,k)\) is the ordinary
Johnson graph.

An AIM workshop report on spectra of matrix families described by graphs considers
the intersection graph of the \(k\)-subsets of an \(n\)-element set, namely
\(J_{\ge1}(n,k)\).  It records the bounds
\begin{equation*}
        n-2k+2
        \le
        \mr^\R\!\left(J_{\ge1}(n,k)\right)
        \le
        n-1
        \qquad
        (2\le k\le n/2),
\end{equation*}
observes that the lower bound is attained when \(k=2\) and when \(k=n/2\),
and asks whether it is always sharp
\cite[Question~Q5]{AIMQuestions}.  In our notation, the question is the
following.

\begin{problem}[\cite{AIMQuestions}]\label{prob:aim}
Let \(n,k\) be integers with \(2\le k\le n/2\).  Is it true that
\[
        \mr^\R\!\left(J_{\ge1}(n,k)\right)=n-2k+2?
\]
\end{problem}

The present paper resolves Problem~\ref{prob:aim} and, more generally, determines
the minimum rank of the entire family \(J_{\ge s}(n,k)\) over
every infinite field, together with its positive semidefinite minimum rank
over \(\R\).

Our main theorem is the following.

\begin{theorem}\label{thm:main}
Let \(n,k,s\) be integers with \(k\ge2\), \(1\le s\le k-1\), and
\(n\ge2k-s\). Then, for every infinite field \(\F\),
\[
\mr^\F\!\left(J_{\ge s}(n,k)\right)
=
\binom{n-2(k-s)}{s}.
\]
Moreover,
\[
\mrp^\R\!\left(J_{\ge s}(n,k)\right)
=
\binom{n-2(k-s)}{s}.
\]
\end{theorem}

\begin{remark}[The remaining values of \(n\)]\label{rem:complete-regime}
The condition \(n\ge2k-s\) marks the boundary of the nontrivial parameter
range.  Indeed, for any \(S,T\in\binom{[n]}{k}\),
\[
        |S\cap T|
        =
        |S|+|T|-|S\cup T|
        \ge 2k-n.
\]
Hence, if \(k<n\le2k-s\), then every two distinct vertices have intersection
of size at least \(s\), and therefore $J_{\ge s}(n,k)=K_{\binom{n}{k}}$. It follows that, for every field \(\F\),
\[
        \mr^\F\!\left(J_{\ge s}(n,k)\right)
        =
        \mrp^\R\!\left(J_{\ge s}(n,k)\right)
        =
        1,
\]
where the second equality is understood over \(\R\).  At \(n=2k-s\), this
agrees with \cref{thm:main}, since
\[
        \binom{n-2(k-s)}{s}=\binom{s}{s}=1.
\]
If \(n=k\), the graph has a single vertex and both parameters are \(0\).
Thus \cref{thm:main} also yields, together with these elementary cases, a
complete determination for every \(n\ge k\).
\end{remark}

For \(s=1\), \cref{thm:main} gives $\mr^\F\left(J_{\ge1}(n,k)\right)=n-2k+2$ for every infinite field \(\F\).  Thus it answers \cref{prob:aim} and
establishes the same formula over every infinite field.  At the opposite
endpoint \(s=k-1\), where \(J_{\ge k-1}(n,k)=J(n,k)\), it gives, for every
infinite field \(\F\),
\[
\mr^\F\!\left(J(n,k)\right) = \mrp^\R\!\left(J(n,k)\right) = \binom{n-2}{k-1}.
\]
For \(n\ge 2k\), this recovers the ordinary and positive
semidefinite minimum-rank values determined by Fallat, Gupta,
Herman, and Parenteau \cite[Theorem~25]{FallatGuptaHermanParenteau2024}.
For \(k<n<2k\), the same conclusion follows from the isomorphism
\(J(n,k)\cong J(n,n-k)\) and an application of the same theorem
with \(d=n-k\).

For graphs without isolated vertices, positive semidefinite minimum rank has
an equivalent formulation in terms of faithful orthogonal representations. We use the
nonzero-vector convention: a \emph{faithful orthogonal representation} of a graph
$G$ in $\R^d$ is an assignment of a nonzero vector $x_v\in\R^d$ to each
vertex $v$ such that, for all distinct $u,v\in V(G)$,
\[
        x_u\cdot x_v=0
        \quad\Longleftrightarrow\quad
        \{u,v\}\notin E(G).
\]
Let $\fod^\R(G)$ denote the least integer $d$ for which $G$ admits such a representation in $\R^d$. 

For graphs without isolated vertices, this definition agrees with the notion of a faithful
orthogonal representation used by Lov\'asz, Saks and Schrijver
\cite{LovaszSaksSchrijver1989}. For ordinary, not necessarily
faithful, orthogonal representations, Haviv likewise uses the convention
that nonadjacent vertices are represented by orthogonal nonzero vectors
\cite{Haviv2019}.  Golovnev and Haviv instead use the complementary
convention, in which orthogonality is imposed on adjacent vertices
\cite{GolovnevHaviv2022}.  Thus statements about orthogonality dimension in
their convention translate to ours by passing to the complement graph.

Under the hypotheses of \cref{thm:main}, the graph
\(J_{\ge s}(n,k)\) has no isolated vertices.  Hence the standard Gram-matrix
correspondence, recalled in Lemma~\ref{lem:fod-psd}, shows that
\[
        \fod^\R\!\left(J_{\ge s}(n,k)\right)
        =
        \mrp^\R\!\left(J_{\ge s}(n,k)\right)
        =
        \binom{n-2(k-s)}{s}.
\]
In particular, if \(n\ge8\) is divisible by \(4\), then
\[
        \fod^\R\!\left(\overline{\KG(n,n/4)}\right)
        =
        \frac{n}{2}+2<n,
\]
thereby answering a question of Golovnev \cite{GolovnevMO}.

\subsection*{Paper organization}

In \cref{sec:lower}, we recall the Gram-matrix correspondence between
faithful orthogonal representations and positive semidefinite minimum
rank, and prove the lower bound in \cref{thm:main} over every field.
In \cref{sec:algebraic}, we establish the matching upper bound over every
infinite field by means of an exterior-power construction, a Lagrange
interpolation identity, and a generic nonvanishing argument.
In \cref{sec:psd}, we choose the parameters in an interlacing
configuration so that the same construction becomes positive semidefinite
over \(\R\), thereby completing the proof of \cref{thm:main}.
Finally, \cref{sec:consequences} records consequences for the case \(s=1\)
and for faithful orthogonality dimension, together with comparisons with
related results and a finite-field obstruction.

\section{Preliminaries and the lower bound}\label{sec:lower}

We record the following standard Gram-factorization correspondence; see, for
example, \cite[Section~1.3]{HogbenPalmowskiRobersonSeverini2017}.  For
completeness, we include a short proof.

\begin{lemma}\label{lem:fod-psd}
If \(G\) is a finite simple graph with no isolated vertices, then
\[
        \fod^\R(G)=\mrp^\R(G).
\]
\end{lemma}

\begin{proof}
A faithful orthogonal representation in \(\R^d\) has a positive semidefinite
Gram matrix in \(\cSp^\R(G)\) of rank at most \(d\).  Hence
\(\mrp^\R(G)\le\fod^\R(G)\).

Conversely, let \(A\in\cSp^\R(G)\) have rank \(r\).  Since \(A\succeq0\),
there exists \(X\in\R^{r\times |V(G)|}\) such that \(A=X^\top X\).
Writing \(x_v\) for the column of \(X\) indexed by \(v\), we have
\(a_{uv}=x_u\cdot x_v\).  For distinct \(u,v\), the prescribed off-diagonal pattern of \(A\) gives
\[
x_u\cdot x_v=0
\quad\Longleftrightarrow\quad
\{u,v\}\notin E(G).
\]  
Since \(G\) has no isolated vertices, each \(x_v\) is
nonzero.  Hence \(\fod^\R(G)\le r\), and minimizing over \(A\) completes the
proof.
\end{proof}

The next proposition gives the lower bound in \cref{thm:main}.  It also
follows from the zero forcing number of the generalized Johnson graph
\(J_{\ge s}(n,k)=J_{\{s,s+1,\ldots,k-1\}}(n,k)\), determined by Abiad,
Simoens, and Zeijlemaker
\cite[Theorem~8]{AbiadSimoensZeijlemaker2024}, together with the standard
zero-forcing bound on maximum nullity
\cite[Proposition~2.4]{AIMZeroForcing2008}.  We include a short direct proof
because it is elementary, uses only the prescribed off-diagonal
zero--nonzero pattern, and makes the field-independence transparent.

\begin{proposition}\label{prop:lower}
Let \(k\ge2\), \(1\le s\le k-1\), and \(n\ge2k-s\).  Then, for every field \(\F\),
\[
        \mr^\F\!\left(J_{\ge s}(n,k)\right)
        \ge
        \binom{n-2(k-s)}{s}.
\]
In particular,
\[
        \mrp^\R\!\left(J_{\ge s}(n,k)\right)
        \ge
        \binom{n-2(k-s)}{s}.
\]
\end{proposition}

\begin{proof}
Set
\[
        t=k-s,
        \qquad
        m=n-2t=n-2(k-s).
\]
The assumption \(n\ge2k-s\) is equivalent to \(m\ge s\). Partition $[n]$ into three
pairwise disjoint sets
\[
        P,\quad Q,\quad R,
        \qquad
        |P|=|Q|=t,
        \quad
        |R|=m.
\]
For each $X\in\binom{R}{s}$, define $S_X=P\cup X$ and $T_X=Q\cup X$. Then $S_X$ and $T_X$ are distinct $k$-subsets of $[n]$ since $t\ge1$.  For $X,Y\in\binom{R}{s}$ one has $S_X\cap T_Y=X\cap Y$, and therefore
\[
        |S_X\cap T_Y|\ge s
        \quad\Longleftrightarrow\quad
        X=Y.
\]

Let $A\in\cS^\F(J_{\ge s}(n,k))$.  The submatrix of $A$ with rows indexed by the
sets $S_X$ and columns indexed by the sets $T_Y$, where
$X,Y\in\binom{R}{s}$, is diagonal.  Its diagonal entries are nonzero because
$S_X$ and $T_X$ are distinct adjacent vertices.  Hence this submatrix has rank
$\binom{m}{s}$, and so
\[
        \rank A\ge \binom{m}{s}
        =
        \binom{\,n-2(k-s)\,}{s}.
\]
Minimizing over $A$ gives the ordinary lower bound.  The positive semidefinite
lower bound follows from
$\cSp^\R(J_{\ge s}(n,k))\subseteq\cS^\R(J_{\ge s}(n,k))$.
\end{proof}

\section{An exterior-power construction over infinite fields}\label{sec:algebraic}

In this section we prove the matching upper bound for the minimum rank over an arbitrary infinite field.  Throughout the section, fix an infinite field $\F$ and integers $n,k,s$ satisfying
\[
        k\ge2,
        \qquad
        1\le s\le k-1,
        \qquad
        n\ge2k-s.
\]
Set
\[
        t=k-s,
        \qquad
        m=n-2t=n-2(k-s).
\]
Thus $t\ge1$ and $m\ge s$.

\subsection{The exterior-power construction}\label{subsec:gram}

We recall the standard exterior-power construction over \(\F\).
For the determinant formula defining the exterior power of a bilinear form,
compare \cite[Definition~3.1]{McGarraghy2002}. 

If \(V\) is a finite-dimensional vector space over \(\F\), then
\(\bigwedge^s V\) is generated by symbols $v_1\wedge\cdots\wedge v_s$, modulo multilinearity and the alternating relations.
If \(\dim V=m\) and \(e_1,\ldots,e_m\) is a basis of \(V\), then $\{e_{i_1}\wedge\cdots\wedge e_{i_s}: 1\le i_1<\cdots<i_s\le m\}$ forms a basis of \(\bigwedge^sV\). Hence $\dim\bigwedge^s V=\binom{m}{s}$.

Fix a basis \(e_1,\ldots,e_m\) of \(V\).  For $I=\{i_1<\cdots<i_s\}\subseteq[m]$, write $e_I=e_{i_1}\wedge\cdots\wedge e_{i_s}$. Every bilinear form \(B\) on \(V\) determines a unique bilinear form
\(B^{\wedge s}\) on \(\bigwedge^sV\) by
\[
B^{\wedge s}(e_I,e_J)
=
\det\bigl(B(e_{i_a},e_{j_b})\bigr)_{a,b=1}^s
\]
for $I=\{i_1<\cdots<i_s\}$, $J=\{j_1<\cdots<j_s\}$ and by bilinear extension. Indeed, the values of a bilinear form on all pairs of basis vectors
determine the form uniquely. If \(B\) is symmetric, then \(B^{\wedge s}\) is symmetric.  Over
\(\R\), if \(B\) is positive definite, then \(B^{\wedge s}\) is
positive definite as well.

Let $\alpha_1,\ldots,\alpha_n,\beta_1,\ldots,\beta_m$ be pairwise distinct elements of $\F$. Define
\[
        q(x)=\prod_{j=1}^n(x-\alpha_j),
        \qquad
        p(x)=\prod_{i=1}^m(x-\beta_i),
\]
and
\[
        w_i=\frac{1}{p'(\beta_i)q(\beta_i)},
        \qquad
        i=1,\ldots,m.
\]
All $w_i$ are nonzero.  Let $B$ be the nondegenerate symmetric bilinear form on
$\F^m$ given by
\[
        B(u,v)=\sum_{i=1}^m w_i u_i v_i.
\]
For each $S\in\binom{[n]}{k}$, put
\[
        f_S(x)=\prod_{j\notin S}(x-\alpha_j).
\]
This choice is arranged so that the quotient
\(f_S(x)f_T(x)/q(x)\) has simple poles precisely at the points
\(\alpha_u\) with \(u\in S\cap T\).
For $0\le a\le s-1$, define $v_{S,a}=\bigl(\beta_i^a f_S(\beta_i)\bigr)_{i=1}^m\in\F^m$, and set
\[
        z_S=v_{S,0}\wedge v_{S,1}\wedge\cdots\wedge v_{S,s-1}
        \in\bigwedge^s\F^m.
\]
Let \(B^{\wedge s}\) be the induced bilinear form on
\(\bigwedge^s\F^m\) defined above.
Define the matrix
\begin{equation}\label{eq:A-definition}
        A=(A_{S,T})_{S,T\in\binom{[n]}{k}},
        \qquad
        A_{S,T}=B^{\wedge s}(z_S,z_T).
\end{equation}
Since \(B\) is symmetric, so is \(B^{\wedge s}\), and hence \(A\) is symmetric.
Equivalently,
\begin{equation}\label{eq:det-formula}
        A_{S,T}=\det\bigl(M^{S,T}_{a,b}\bigr)_{a,b=0}^{s-1},
\end{equation}
where
\begin{equation}\label{eq:M-entry}
        M^{S,T}_{a,b}
        =B(v_{S,a},v_{T,b})
        =
        \sum_{i=1}^m
        \frac{\beta_i^{a+b}f_S(\beta_i)f_T(\beta_i)}
             {p'(\beta_i)q(\beta_i)}.
\end{equation}
Since $A$ is a Gram matrix of vectors in $\bigwedge^s\F^m$ with respect to the
bilinear form $B^{\wedge s}$,
\begin{equation}\label{eq:rank-upper-general}
        \rank A\le \dim\bigwedge^s\F^m=\binom{m}{s}.
\end{equation}
The remaining issue is to choose the parameters so that the off-diagonal
zero--nonzero pattern of $A$ is exactly that of $J_{\ge s}(n,k)$.

\subsection{A Lagrange interpolation identity}
\label{subsec:interpolation}
The following identity expresses $M^{S,T}_{a,b}$ in terms of
\(S\cap T\), and will yield the required vanishing of \(A_{S,T}\)
when \(|S\cap T|<s\).

\begin{lemma}\label{lem:moment-identity}
Let $S,T\in\binom{[n]}{k}$. Put $U=S\cap T$ and $W=[n]\setminus(S\cup T)$. For $u\in U$, define\footnote{This is well defined because all the \(\alpha_j\)'s and \(\beta_i\)'s
are pairwise distinct.}
\begin{equation}\label{eq:lambda-u}
        \lambda_u
        =
        \frac{\prod_{j\in W}(\alpha_u-\alpha_j)}
             {p(\alpha_u)
              \prod_{v\in U\setminus\{u\}}(\alpha_u-\alpha_v)}.
\end{equation}
Then, for $0\le a,b\le s-1$,
\begin{equation}\label{eq:moment-identity}
        M^{S,T}_{a,b}
        =
        -\sum_{u\in U}\lambda_u\alpha_u^{a+b}.
\end{equation}
If $U=\varnothing$, the sum is interpreted as $0$.
\end{lemma}

\begin{proof}
Set $R_{S,T}(x) := f_S(x)f_T(x)/q(x)$. Thus we obtain
\[
        R_{S,T}(x) = \frac{\left(\prod_{j\notin S}(x-\alpha_j)\right)\left(\prod_{j\notin T}(x-\alpha_j)\right)}
             {\prod_{j\in [n]}(x-\alpha_j)}
        =
        \frac{\prod_{j\in W}(x-\alpha_j)}
             {\prod_{u\in U}(x-\alpha_u)}.
\]
Moreover
\[
        |W|=n-2k+|U|=m-2s+|U|.
\]

By \eqref{eq:M-entry},
\[
        M^{S,T}_{a,b}
        =
        \sum_{i=1}^m
        \frac{\beta_i^{a+b}R_{S,T}(\beta_i)}
             {p'(\beta_i)}.
\]
Consider the rational function
\[
        H_{a,b}(x)
        :=
        \frac{x^{a+b}R_{S,T}(x)}{p(x)}
        =
        \frac{x^{a+b}\prod_{j\in W}(x-\alpha_j)}
             {p(x)\prod_{u\in U}(x-\alpha_u)}.
\]
The degree of its numerator minus the degree of its denominator is
\[
        a+b+|W|-m-|U|
        =
        a+b-2s
        \le -2.
\]
Consequently, the formal Laurent expansion of \(H_{a,b}(x)\) in
\(\F((x^{-1}))\) has zero coefficient of \(x^{-1}\).

All finite poles of \(H_{a,b}\) are simple and lie among the points
\(\beta_i\), \(1\le i\le m\), and \(\alpha_u\), \(u\in U\). In the
partial-fraction decomposition of \(H_{a,b}\), the coefficient of
\((x-\beta_i)^{-1}\) is $\beta_i^{a+b}R_{S,T}(\beta_i)/p'(\beta_i)$,
and the sum of these coefficients is \(M^{S,T}_{a,b}\). Similarly,
the coefficient of $(x-\alpha_u)^{-1}$ is
$\lambda_u\alpha_u^{a+b}$, with $\lambda_u$ as in \eqref{eq:lambda-u}.  Hence
\[
        M^{S,T}_{a,b}+\sum_{u\in U}\lambda_u\alpha_u^{a+b}=0,
\]
which is \eqref{eq:moment-identity}.
\end{proof}

\begin{lemma}\label{lem:vanishing}
If $S,T\in\binom{[n]}{k}$ satisfy $|S\cap T|<s$, then
$A_{S,T}=0$.
\end{lemma}

\begin{proof}
Let \(U=S\cap T\). By Lemma~\ref{lem:moment-identity}, the matrix
\((M^{S,T}_{a,b})_{a,b=0}^{s-1}\) factors as $-V_U\Lambda_U V_U^{\top}$, where
\[
        V_U=(\alpha_u^a)_{0\le a\le s-1,\ u\in U},
        \qquad
        \Lambda_U=\operatorname{diag}(\lambda_u:u\in U).
\]
Its rank is at most \(|U|<s\). Therefore, its determinant is zero, and
\eqref{eq:det-formula} gives \(A_{S,T}=0\).
\end{proof}

\subsection{Generic nonvanishing on edges}\label{subsec:nonvanishing}

It remains to show that the parameters can be chosen so that $A_{S,T}$ is
nonzero for every distinct pair $S,T$ with $|S\cap T|\ge s$.  We first prove
that each such entry is not identically zero as a rational function of the
parameters.

\begin{lemma}\label{lem:generic-nonzero}
Let $\F$ be any field, and regard
$\alpha_1,\ldots,\alpha_n,\beta_1,\ldots,\beta_m$ as algebraically independent
variables over $\F$.  If $S,T\in\binom{[n]}{k}$ satisfy $|S\cap T|\ge s$, then
\[
        \Delta_{S,T}(\alpha,\beta)
        :=
        \det\bigl(M^{S,T}_{a,b}\bigr)_{a,b=0}^{s-1}
\]
is not the zero rational function in
$\F(\alpha_1,\ldots,\alpha_n,\beta_1,\ldots,\beta_m)$.
\end{lemma}

\begin{proof}
Let $U=S\cap T$ and $W=[n]\setminus(S\cup T)$. By Lemma \ref{lem:moment-identity}, $(M^{S,T}_{a,b})_{a,b=0}^{s-1}=-V_U\Lambda_U V_U^{\top}$, where
\[
        V_U=(\alpha_u^a)_{0\le a\le s-1,\ u\in U},
        \qquad
        \Lambda_U=\operatorname{diag}(\lambda_u:u\in U).
\]
The Cauchy--Binet formula gives
\begin{equation}\label{eq:cauchy-binet-edge}
        \Delta_{S,T}
        =
        (-1)^s
        \sum_{X\in\binom{U}{s}}
        \left(\prod_{u\in X}\lambda_u\right)
        \prod_{\substack{u,v\in X\\u<v}}(\alpha_v-\alpha_u)^2.
\end{equation}
Choose one set $X_0\in\binom{U}{s}$.  Since $m\ge s$, choose an injection
\[
        \eta:X_0\hookrightarrow [m].
\]
Introduce independent variables $\eps_u$, $u\in X_0$, and make the
specialization
\[
        \beta_{\eta(u)}=\alpha_u+\eps_u
        \qquad (u\in X_0),
\]
leaving all other $\alpha$'s and $\beta$'s algebraically independent.  View the
right-hand side of \eqref{eq:cauchy-binet-edge} as a Laurent series in the
variables $\eps_u$.

The summand indexed by $X_0$ has a Laurent term with pole part
\[
        \prod_{u\in X_0}\eps_u^{-1}.
\]
Indeed, from \eqref{eq:lambda-u}, for each $u\in X_0$, the denominator factor $p(\alpha_u)$ in
$\lambda_u$ contains
\[
        \alpha_u-\beta_{\eta(u)}=-\eps_u.
\]
After multiplying the \(X_0\)-summand by
\(\prod_{u\in X_0}\varepsilon_u\) and then setting all
\(\varepsilon_u=0\), we obtain
\[
\left(
\prod_{u\in X_0}
\frac{
\prod_{j\in W}(\alpha_u-\alpha_j)
}{
\left(\prod_{v\in[m]\setminus\{\eta(u)\}}
(\alpha_u-\beta_v)\right)
\left(\prod_{v\in U\setminus\{u\}}
(\alpha_u-\alpha_v)\right)
}
\right)
\prod_{\substack{\{u,v\}\subseteq X_0\\u<v}}
(\alpha_v-\alpha_u)^2.
\]
This is a nonzero rational function in the remaining
algebraically independent variables.

No summand indexed by $X\ne X_0$ contains the same full pole part
$\prod_{u\in X_0}\eps_u^{-1}$.  Since $|X|=|X_0|=s$ and $X\ne X_0$, the set
$X$ omits at least one element of $X_0$, and therefore the corresponding
summand has no pole in the omitted variable.  Hence the coefficient of
$\prod_{u\in X_0}\eps_u^{-1}$ in the specialized Laurent series is nonzero.
Thus $\Delta_{S,T}$ is not the zero rational function.
\end{proof}

\begin{proposition}\label{prop:ordinary-upper}
Let \(k\ge2\), \(1\le s\le k-1\), and \(n\ge2k-s\). Then, for every infinite field \(\F\),
\[
        \mr^\F\!\left(J_{\ge s}(n,k)\right)
        \le
        \binom{\,n-2(k-s)\,}{s}.
\]
\end{proposition}

\begin{proof}

For each pair of distinct adjacent vertices $S,T$ of
$J_{\ge s}(n,k)$, choose a nonzero polynomial $D_{S,T}$ that clears all
denominators in the expression for $\Delta_{S,T}$ in
\eqref{eq:cauchy-binet-edge}, and set
$N_{S,T}=D_{S,T}\Delta_{S,T}$. By Lemma \ref{lem:generic-nonzero}, $N_{S,T}$ is a nonzero
polynomial.  Consider the finite product
\[
        \Phi(\alpha,\beta)
        =
        \prod_{1\le a<b\le n}(\alpha_a-\alpha_b)
        \prod_{1\le i<j\le m}(\beta_i-\beta_j)
        \prod_{\substack{1\le i\le m\\1\le a\le n}}(\beta_i-\alpha_a)
        \prod_{\substack{S,T\in\binom{[n]}{k}\\ S\ne T,\ |S\cap T|\ge s}}
        N_{S,T}(\alpha,\beta).
\]
This is a nonzero polynomial over $\F$.  Since $\F$ is infinite, there is a
point $(\alpha,\beta)\in\F^{n+m}$ at which $\Phi$ is nonzero.

For this choice, the elements $\alpha_1,\ldots,\alpha_n,\beta_1,\ldots,\beta_m$ are pairwise distinct.  The matrix $A$ in~\eqref{eq:A-definition} has rank at most $\binom{m}{s}$ by~\eqref{eq:rank-upper-general}.  By Lemma~\ref{lem:vanishing}, $A_{S,T}=0$ whenever $|S\cap T|<s$.  By the choice of $\Phi$, $A_{S,T}\ne0$ whenever $S\ne T$ and $|S\cap T|\ge s$. Hence $ A\in\cS^\F(J_{\ge s}(n,k))$ and
\[
        \rank A\le \binom{m}{s}
        =
        \binom{\,n-2(k-s)\,}{s}.
\]
This proves the desired upper bound.
\end{proof}

\section{The positive semidefinite construction over the reals}\label{sec:psd}

The construction above becomes positive semidefinite over $\R$ whenever all
weights $w_i$ are positive.  We ensure this by an interlacing choice of the
parameters.

Continue to write $t=k-s$ and $m=n-2t$.  Let $\Omega$ be the open subset of
$\R^{n+m}$ consisting of tuples
$(\alpha_1,\ldots,\alpha_n,\beta_1,\ldots,\beta_m)$ satisfying
\begin{equation}\label{eq:interlacing}
        \alpha_1<\beta_1<\alpha_2<\beta_2<\cdots
        <\alpha_m<\beta_m<\alpha_{m+1}<\cdots<\alpha_n.
\end{equation}

This open set is nonempty because $n-m=2t\ge2$.  The strict inequalities in
\eqref{eq:interlacing} will be used only through the following sign
computation.

\begin{lemma}\label{lem:positive-weights}
For every point of $\Omega$, we have
\[
        \frac{1}{p'(\beta_i)q(\beta_i)}>0
        \qquad
        (i=1,\ldots,m).
\]
\end{lemma}

\begin{proof}
Since $\beta_i$ lies between $\alpha_i$ and $\alpha_{i+1}$, exactly $n-i$ of
the factors $\beta_i-\alpha_a$ are negative.  Hence
\[
        \sgn q(\beta_i)=(-1)^{n-i}.
\]
Similarly, exactly $m-i$ of the factors $\beta_i-\beta_j$, $j\ne i$, are
negative, so
\[
        \sgn p'(\beta_i)=(-1)^{m-i}.
\]
Because $n-m=2t$ is even, these two signs are equal.  Thus
$p'(\beta_i)q(\beta_i)>0$.
\end{proof}

\begin{proposition}\label{prop:psd-upper}
Let \(k\ge2\), \(1\le s\le k-1\), and \(n\ge2k-s\). Then
\[
        \mrp^\R\!\left(J_{\ge s}(n,k)\right)
        \le
        \binom{\,n-2(k-s)\,}{s}.
\]
\end{proposition}

\begin{proof}
For each pair of distinct adjacent vertices $S,T$ of $J_{\ge s}(n,k)$, let $N_{S,T}$ be the nonzero
polynomial numerator used in the proof of Proposition \ref{prop:ordinary-upper}.  The
finite product of all $N_{S,T}$ is a nonzero real polynomial. Since a nonzero real polynomial cannot vanish on a nonempty open subset of
$\R^{n+m}$, its nonzero locus meets $\Omega$.  We may therefore choose a
point of $\Omega$ at which every $N_{S,T}$ is nonzero.

For this choice of parameters, Lemma \ref{lem:positive-weights} shows that all weights $w_i$ are positive.  Hence the bilinear form
\[
        B(u,v)=\sum_{i=1}^m w_i u_i v_i
\]
on $\R^m$ is positive definite, and its induced exterior-power form
$B^{\wedge s}$ on $\bigwedge^s\R^m$ is positive definite.  Therefore the
matrix
\[
        A=(B^{\wedge s}(z_S,z_T))_{S,T\in\binom{[n]}{k}}
\]
is positive semidefinite.

The same vanishing and generic nonvanishing arguments used in
\cref{sec:algebraic} show that, for distinct $S,T$,
\[
        A_{S,T}=0
        \quad\Longleftrightarrow\quad
        |S\cap T|<s.
\]
Thus $A\in\cSp^\R(J_{\ge s}(n,k))$.  By \eqref{eq:rank-upper-general},
\[
        \rank A\le \binom{m}{s}
        =
        \binom{\,n-2(k-s)\,}{s}.
\]
This gives the positive semidefinite upper bound.
\end{proof}

We are now ready to prove the main theorem.

\begin{proof}[Proof of \cref{thm:main}]
The ordinary lower bound is Proposition \ref{prop:lower}, and the ordinary upper bound over
every infinite field is Proposition \ref{prop:ordinary-upper}.  Hence
\[
        \mr^\F\!\left(J_{\ge s}(n,k)\right)
        =
        \binom{\,n-2(k-s)\,}{s}
\]
for every infinite field $\F$.  Over $\R$, Proposition \ref{prop:lower} also gives the
positive semidefinite lower bound, and Proposition \ref{prop:psd-upper} gives the matching
positive semidefinite upper bound.  Therefore
\[
        \mrp^\R\!\left(J_{\ge s}(n,k)\right)
        =
        \binom{\,n-2(k-s)\,}{s}.
\qedhere\]
\end{proof}

\section{Consequences and related remarks}\label{sec:consequences}
Taking \(s=1\) in \cref{thm:main} gives an affirmative answer to
\cref{prob:aim}.

\begin{corollary}\label{cor:aim}
Let \(n,k\) be integers with \(2\le k\le n/2\). Then, for every infinite
field \(\F\), we have
\[
        \mr^\F\!\left(J_{\ge1}(n,k)\right)=\mrp^\R\!\left(J_{\ge1}(n,k)\right)=n-2k+2.
\]
\end{corollary}

\begin{proof}
Apply \cref{thm:main} with \(s=1\).
\end{proof}

\begin{corollary}\label{cor:fod-threshold}
Let \(k\ge2\), \(1\le s\le k-1\), and \(n\ge2k-s\). Then
\[
        \fod^\R\!\left(J_{\ge s}(n,k)\right)
        =
        \binom{\,n-2(k-s)\,}{s}.
\]
In particular, if \(n\ge8\) is divisible by \(4\), then
\[
        \fod^\R\!\left(\overline{\KG(n,n/4)}\right)
        =
        \frac n2+2
        <
        n.
\]
\end{corollary}

\begin{proof}
The graph \(J_{\ge s}(n,k)\) has no isolated vertices. Indeed, since
\(n\ge2k-s\) and \(s\le k-1\), we have \(n\ge k+1\). For any
\(S\in\binom{[n]}{k}\), choose \(a\in S\) and \(b\in[n]\setminus S\), and set $T=(S\setminus\{a\})\cup\{b\}$. Then \(T\ne S\) and
\[
        |S\cap T|=k-1\ge s,
\]
so \(T\) is adjacent to \(S\).

The first assertion now follows from \cref{lem:fod-psd,thm:main}. For the
final statement, take \(s=1\) and \(k=n/4\). Since
\(J_{\ge1}(n,n/4)=\overline{\KG(n,n/4)}\), we obtain
\[
        \fod^\R\!\left(\overline{\KG(n,n/4)}\right)
        =
        n-2\left(\frac n4-1\right)
        =
        \frac n2+2.
\]
Since \(n\ge8\), this is strictly smaller than \(n\).
\end{proof}

\begin{remark}
Haviv, using the same nonedge-orthogonality convention, proved that the
ordinary, not necessarily faithful, real orthogonality dimension of
\(\overline{\KG(n,n/4)}\) is \(n/2+2\) whenever \(4\mid n\); see
\cite[Corollary~3.3]{Haviv2019}. Together with
\cref{cor:fod-threshold}, this shows that, for every \(n\ge8\) divisible
by \(4\), requiring the representation to be faithful does not increase
the minimum dimension. In particular, this answers the question posed
in \cite{GolovnevMO}.
\end{remark}

\begin{remark}[A nonsymmetric counterpart]
For \(s=1\), Berkman and Haviv proved that, for every \(n\ge2k\), there
exists an integer \(c=c(n,k)\) such that, over every field \(\F\) with
\(|\F|\ge c\), the Kneser graph \(\KG(n,k)\) admits a faithful
\((n-2k+2)\)-dimensional independent representation
\cite[Theorem~6.4 in the full version]{BerkmanHaviv2025}. By their matrix formulation
\cite[Remark~3.10 in the full version]{BerkmanHaviv2025}, this yields a matrix
\[
        M\in\F^{\binom{n}{k}\times\binom{n}{k}}
\]
of rank at most \(n-2k+2\), not necessarily symmetric, whose diagonal
entries are nonzero and such that, for all distinct
\(S,T\in\binom{[n]}{k}\),
\[
        M_{S,T}=0
        \quad\Longleftrightarrow\quad
        S\cap T=\varnothing.
\]
Equivalently, the off-diagonal zero--nonzero pattern of \(M\) is precisely
that prescribed by \(J_{\ge1}(n,k)\). By contrast,
\cref{cor:aim} determines the exact minimum rank under the additional
symmetry requirement over every infinite field, and also determines the
positive semidefinite minimum rank over \(\R\).
\end{remark}

\begin{remark}[Finite fields]\label{rem:finite-fields}
The hypothesis that \(\F\) is infinite cannot be removed.  Indeed, over
\(\mathbb F_2\) one has
\[
        \mr^{\mathbb F_2}(J_{\ge1}(5,2))=4,
\]
whereas \(5-2\cdot 2+2=3<4\).

To see this, identify the vertices of \(J_{\ge1}(5,2)\) with the edges of
\(K_5\).  Over \(\mathbb F_2\), every matrix in
\(\cS^{\mathbb F_2}(J_{\ge1}(5,2))\) has the same prescribed off-diagonal
entries: the \((e,f)\)-entry is \(1\) exactly when the two edges \(e\) and
\(f\) of \(K_5\) meet.  The diagonal entries are arbitrary.

We first show that every such matrix has rank at least \(4\).  The ten
columns must be pairwise distinct.  Indeed, if \(e\ne f\) are two edges of
\(K_5\), then there is an edge \(g\) of \(K_5\), distinct from both \(e\)
and \(f\), which meets exactly one of \(e\) and \(f\).  Hence the \(g\)-th
entries of the \(e\)-th and \(f\)-th columns are different.  A vector space
of dimension at most \(3\) over \(\mathbb F_2\) has at most \(2^3=8\)
vectors, so a matrix with ten pairwise distinct columns must have rank at
least \(4\).

Conversely, let \(N\) be the \(5\times 10\) vertex-edge incidence matrix of
\(K_5\) over \(\mathbb F_2\).  Then \(N^{\top} N\) has zero diagonal and has
off-diagonal entry \(1\) precisely for intersecting pairs of edges.  Thus
\(N^{\top} N\in \cS^{\mathbb F_2}(J_{\ge1}(5,2))\).  Since each column of \(N\) has two
ones, the sum of the rows of \(N\) is zero, and hence \(\rank N\le4\).
Therefore
\[
        \rank(N^{\top} N)\le \rank N\le4.
\]
Combining the two inequalities gives $\mr^{\mathbb F_2}(J_{\ge1}(5,2))=4$.
\end{remark}

\section*{Acknowledgments}

We thank Ishay Haviv for helpful comments on an earlier version of this
manuscript. His comments prompted us to consider the positive semidefinite
variant of the minimum-rank result.

\section*{Statements and Declarations}

\subsection*{Funding}

The authors did not receive support from any organization for the submitted
work.

\subsection*{Competing interests}

The authors have no relevant financial or non-financial interests to disclose.

\subsection*{Data availability}

No datasets were generated or analyzed during the current study.

\subsection*{Use of generative AI}

We used ChatGPT to generate exploratory code for assessing the plausibility of
conjectured statements and to help polish the language of proof drafts written
by the authors. All mathematical claims, calculations, and proofs were
independently verified by the authors. The core mathematical ideas and final
arguments were developed by the authors; any suggestions from the tool were
substantially modified or discarded as appropriate. The authors take full
responsibility for the correctness and originality of the paper.

\end{document}